%% file: main.tex
\documentclass[letterpaper, 10 pt, conference]{ieeeconf}  

\IEEEoverridecommandlockouts                              
\usepackage{amsmath}
\usepackage{graphicx}
\usepackage[colorinlistoftodos]{todonotes}
\usepackage[colorlinks=true, allcolors=blue]{hyperref}
\usepackage{amssymb}
\usepackage{cite}
\usepackage{enumerate}
\usepackage{mathtools}

\usepackage{algorithm}
\usepackage[noend]{algpseudocode}

\makeatletter
\def\BState{\State\hskip-\ALG@thistlm}
\makeatother

\newtheorem{theorem}{Theorem}
\newtheorem{remark}{Remark}

\DeclareMathOperator*{\dom}{\mathbf{dom}}
\DeclareMathOperator*{\argmin}{arg\,min}
\DeclareMathOperator*{\argmax}{arg\,max}
\DeclareMathOperator*{\arginf}{arg\,inf}

\graphicspath{ {./pics/} }

\title{\LARGE \bf
On the Set of Possible Minimizers of a Sum of \\Known and Unknown Functions}

\author{Kananart Kuwaranancharoen and Shreyas Sundaram 
\thanks{This research was supported by NSF CAREER award
1653648.  The authors are with the School of Electrical and Computer Engineering at Purdue University.  Email: {\tt \{kkuwaran,sundara2\}@purdue.edu}.}
}

\begin{document}

\maketitle
\thispagestyle{empty}
\pagestyle{empty}

\begin{abstract}
The problem of finding the minimizer of a sum of convex functions is central to the field of optimization. Thus, it is of interest to understand how that minimizer is related to the properties of the individual functions in the sum. In this paper, we consider the scenario where one of the individual functions in the sum is not known completely. Instead, only a region containing the minimizer of the unknown function is known, along with some  general characteristics (such as strong convexity parameters).  Given this limited information about a portion of the overall function, we provide a necessary condition which can be used to construct an upper bound on the region containing the minimizer of the sum of known and unknown functions. We provide this necessary condition in both the general case where the uncertainty region of the minimizer of the  unknown function is arbitrary, and in the specific case where the uncertainty region is a ball.
\end{abstract}

\input{intro}
\section{Notation and Preliminaries}

\subsection{Sets}
We denote the closure, interior, and boundary of a set $\mathcal{E}$ by $\bar{\mathcal{E}}$, ${\mathcal{E}}^\circ$, and $\partial \mathcal{E} = \bar{\mathcal{E}} \setminus {\mathcal{E}}^\circ$, respectively. 

\subsection{Linear Algebra}
We denote by $\mathbb{R}^n$ the $n$-dimensional Euclidean space. For simplicity, we often use $(x_1, \ldots, x_n)$ to represent the column vector 
$\begin{bmatrix}
x_1 & x_2 & \ldots & x_n
\end{bmatrix}^T$. We use $e_i$ to denote the $i$-th basis vector (the vector of all zeros except for a one in the $i$-th position). We denote by $\langle u, v \rangle$ the Euclidean inner product of $u$ and $v$ i.e., $\langle u, v \rangle = u^Tv$, by $\Vert \cdot \Vert$ the Euclidean norm $\lVert x \rVert:=(\sum_i x_i^2)^{1/2}$ and by $\angle (u,v)$ the angle between vectors $u$ and $v$. Note that 
\begin{align*}
    \angle (u,v) = \arccos \bigg( \frac{ \langle u, v \rangle }{\Vert u \Vert \Vert v \Vert}  \bigg).
\end{align*}
We use $\mathcal{B}(x_0, r) = \{ x \in \mathbb{R}^n: \Vert x-x_0 \Vert < r  \}$
and $\bar{\mathcal{B}}(x_0, r)$ to denote the open and closed ball, respectively, centered at $x_0$ of radius $r$. Moreover, the function $u(x_1, x_2): (\mathbb{R}^n \times \mathbb{R}^n) \setminus \{ (z_1,z_2) \in \mathbb{R}^n \times \mathbb{R}^n : z_1 = z_2 \}  \rightarrow \mathbb{R}^n$ denotes the unit vector in the direction of $x_1-x_2$, i.e.,
\begin{align}
    u(x_1, x_2) = \frac{x_1 - x_2}{\Vert x_1 - x_2 \Vert} \quad \text{with} \;\; x_1 \neq x_2. \label{def: unitvec}
\end{align}

\subsection{Convex Sets and Convex Functions}
A set $\mathcal{C}$ in $\mathbb{R}^n$ is said to be convex if, for all $x_1$ and $x_2$ in $\mathcal{C}$ and all $\theta$ in the interval $(0, 1)$, the point 
$(1- \theta) x_1 + \theta x_2 \in \mathcal{C}$. 

We say a vector $g \in \mathbb{R}^n$ is a subgradient of $f : \mathbb{R}^n \rightarrow \mathbb{R}$ at $x \in \dom f$ if for all $z \in \dom f$, $f(z) \geq f(x) + \langle g, z-x \rangle$. 

If $f$ is convex and differentiable, then its gradient at $x$ is a subgradient; however, a subgradient can exist even when $f$ is not differentiable at $x$. A function $f$ is called subdifferentiable at $x$ if there exists at least one subgradient at $x$. The set of subgradients of $f$ at the point $x$ is called the subdifferential of $f$ at $x$, and is denoted $\partial f(x)$. The subdifferential $\partial f(x)$ is always a closed convex set, even if $f$ is not convex. In addition, if $f$ is continuous at $x$, then the subdifferential $ \partial f(x)$ is bounded.

A function $f$ is called strongly convex with parameter $\sigma > 0$ (or $\sigma$-strongly convex) if for all points $x, y \in \dom f$,
$\langle g_x - g_y, x-y \rangle \geq \sigma \Vert x-y \Vert^2$ 
for all $g_x \in \partial f(x)$ and $g_y \in \partial f(y)$. We denote the set of all convex functions by $\mathcal{F}$, and the set of all $\sigma$-strongly convex functions with minimizer $x^*_u$ in the set $\mathcal{A} \subseteq \mathbb{R}^n$ and $\dom (\cdot) = \mathbb{R}^n$ by $\mathcal{S} (\mathcal{A}, \sigma) $.


\section{Problem Statement}
We consider a function of the form
\begin{align}
    f(x) = f^{k}(x) + f^{u}(x) \label{eqn: known+unknown}
\end{align}
where $f^k$ and $f^u$ are convex functions. We assume that we know $f^k$ exactly, but do not know $f^u$,  other than some general properties described below.

We assume that $f^k \in \mathcal{F}$ and $f^u \in \mathcal{S} (\mathcal{A}, \sigma)$ where $\mathcal{A}$ is a compact set (i.e., we only know that $f^u$ is $\sigma$-strongly convex and that its minimizer lies in some set $\mathcal{A}$). Our goal is to find the set of points $x \in \mathbb{R}^n$ that could potentially be the minimizer of $f(x)$ in \eqref{eqn: known+unknown}. To this end, we will seek to characterize the region 
\begin{multline}
    \mathcal{M} (f^k, \mathcal{A}, \sigma) \triangleq \big\{ x \in \mathbb{R}^n : \; \exists f^u \in \mathcal{S} (\mathcal{A}, \sigma), \\
    \mathbf{0} \in \partial f^k (x) + \partial f^u (x) \big\}.
\end{multline}
For simplicity of notation, we will omit the argument of the set $\mathcal{M} (f^k, \mathcal{A}, \sigma)$ and write it as $\mathcal{M}$. Note that $\mathcal{M}$ contains all points $x \in \mathbb{R}^n$ that can potentially be a minimizer of $f$, given $f^k$, and the quantity $\sigma$ and the set $\mathcal{A}$ pertaining to $f^u$. 

\begin{remark}
Returning to the regression scenario involving data that is not directly available to the optimizing entity (described in the Introduction), the unknown function would be of the form $f^u(x) = \| Ax - y \|^2$ where $A$ is a matrix containing (unknown) training data and $y$ is the (unknown) vector of corresponding labels. When $A$ has full rank, the loss function is strongly convex. In addition, if some general underlying statistical properties of the data are known to the optimizing entity, it could estimate a lower bound on the strong convexity parameter $\sigma$, and a region containing the possible minimizer of $f^u(x)$.  Thus, using this information, the central entity seeks to find the set of possible minimizers of the sum of this unknown function and its own loss function (corresponding to data that it has access to directly).
\end{remark}


\section{Analysis for General $\mathcal{A}$}

In this section, we provide a necessary condition for a point $x^* \in \dom f^k$ to be the minimizer of $f$ in the general case where the uncertainty region $\mathcal{A}$ of the minimizer of the unknown function is compact, but of arbitrary shape.

For any given point $x^* \in \mathbb{R}^n \setminus \mathcal{A}$, define the set 
\begin{align}
    \tilde{\mathcal{A}}(\mathcal{A}, x^*) \triangleq \big\{ x \in \partial \mathcal{A}: (1-\theta)x + \theta x^* \notin \mathcal{A}, \;\; \forall \theta \in (0,1) \big\}. \label{set: A_tilde}
\end{align}
In words, $\tilde{\mathcal{A}}(\mathcal{A}, x^*)$ is the set of points $x$ on the boundary of $\mathcal{A}$ such that the line joining $x$ to $x^*$ does not intersect $\mathcal{A}$ (except at $x$).

\begin{theorem} \label{thm: nc_general}
Suppose $f^k \in \mathcal{F}$ and $\mathcal{A} \subseteq \dom f^k$ is a compact set. A necessary condition for a point $x^* \in \mathbb{R}^n$ to be in $\mathcal{M} (f^k, \mathcal{A}, \sigma) \setminus \mathcal{A}$ is 
\begin{align}
    \min_{x_u^* \in \mathcal{A}, \; g^k_{x^*} \in \partial f^k(x^*)} \frac{ \langle g^k_{x^*} , u(x^*, x_u^*) \rangle }{\Vert x^*-x_u^* \Vert} \leq - \sigma. \label{eqn: result0}
\end{align}
Furthermore, the above inequality \eqref{eqn: result0} can be reduced to 
\begin{align}
    \min_{ ( x_u^*, g^k ) \in \mathcal{X}(f^k, \mathcal{A}, x^*) } \frac{ \langle g^k , u(x^*, x_u^*) \rangle }{\Vert x^*-x_u^* \Vert} \leq - \sigma  \label{eqn: result1}
\end{align}
where 
\begin{multline}
    \mathcal{X}(f^k, \mathcal{A}, x^*) \triangleq \big\{ (x, g) \in \tilde{\mathcal{A}}(\mathcal{A}, x^*) \times \partial f^k (x^*) : \\
    \; \langle g , u(x^*, x) \rangle < 0 \big\}. \label{set: X}
\end{multline}
\end{theorem}

\begin{proof}
Suppose $f^u \in \mathcal{S} (\mathcal{A}, \sigma)$. For any $x$, $y \in \mathbb{R}^n$, let $g_x^u \in \partial f^u (x)$ and $g_y^u \in \partial f^u (y)$.
From the definition of a strongly convex function, we have
\begin{equation*}
\langle g_x^u - g_{y}^u,  x-y \rangle \geq \sigma \Vert x-y \Vert^2
\end{equation*}
for all $x, y \in \mathbb{R}^n$. 

Let $x_u^* \in \mathcal{A}$ be the true minimizer of $f^u$ and suppose $x^*$ is the minimizer of $f = f^k + f^u$. Then, substitute $x^*$ into $x$ and $x_u^*$ into $y$ to get
\begin{align*}
\langle g_{x^*}^u - g_{x_u^*}^u, x^*-x_u^* \rangle &\geq \sigma \Vert x^*-x_u^* \Vert^2
\end{align*}
for all $g^u_{x^*} \in \partial f^u(x^*)$ and $g^u_{x_u^*} \in \partial f^u(x_u^*)$. Since $x_u^*$ is the minimizer of $f^u$, we have $\mathbf{0} \in \partial f^u(x_u^*)$. Consider $x^* \notin \mathcal{A}$, which implies $x^* \neq x_u^*$, and rewrite the inequality above (with $g^u_{x_u^*} = \mathbf{0}$) to get
\begin{align*}
\Big\langle g_{x^*}^u,  \frac{x^*-x_u^*}{\Vert x^*-x_u^* \Vert} \Big\rangle \geq \sigma \Vert x^*-x_u^* \Vert > 0. 
\end{align*}
Recall the definition of $u(\cdot, \cdot)$ in \eqref{def: unitvec}. The inequality above becomes      
\begin{align}
    \langle g_{x^*}^u , u(x^*, x_u^*) \rangle \geq \sigma \Vert x^*-x_u^* \Vert. \label{eqn: angle}
\end{align}
Using the fact that $x^*$ is the minimizer of $f = f^k + f^u$, we get $\mathbf{0} \in \partial f^k (x^*) + \partial f^u (x^*)$, so there exists $g^k_{x^*} \in \partial f^k(x^*)$ and $g^u_{x^*} \in \partial f^u(x^*)$ such that $g^k_{x^*} + g_{x^*}^u = \mathbf{0}$. Since the inequality \eqref{eqn: angle} is true for any $g_{x^*}^u \in \partial f^u(x^*)$, we can apply $g_{x^*}^u = - g^k_{x^*}$ to \eqref{eqn: angle} and get 
\begin{align*}
    \langle - g^k_{x^*} , u(x^*, x_u^*) \rangle &\geq \sigma \Vert x^*-x_u^* \Vert, \nonumber \\
    \Leftrightarrow \quad \frac{ \langle g^k_{x^*} , u(x^*, x_u^*) \rangle }{\Vert x^*-x_u^* \Vert} &\leq - \sigma.  
\end{align*}
Thus, if $f^u \in \mathcal{S} (\mathcal{A}, \sigma)$, we have a necessary condition that
\begin{multline*}
    \emph{if} \quad x^* \in \mathcal{M} \setminus \mathcal{A} \quad \emph{then there exist} \quad  x_u^* \in \mathcal{A}  \quad  \emph{and} \\
      \quad g^k_{x^*} \in \partial f^k(x^*) \quad  \emph{such that} \quad \frac{ \langle g^k_{x^*} , u(x^*, x_u^*) \rangle }{\Vert x^*-x_u^* \Vert} \leq - \sigma. 
\end{multline*}
Since the sets $\mathcal{A}$ and $\partial f^k(x^*)$ are compact by the assumption that $f^k$ is convex, the necessary condition above is equivalent to
\begin{align}
    \min_{x_u^* \in \mathcal{A}, \; g^k_{x^*} \in \partial f^k(x^*)} \frac{ \langle g^k_{x^*} , u(x^*, x_u^*) \rangle }{\Vert x^*-x_u^* \Vert} \leq - \sigma. \label{eqn: nec_1}
\end{align}
Next, we will show that we can consider the minimum over the set $\mathcal{X}$ (defined in \eqref{set: X}) instead of $\mathcal{A} \times \partial f^k(x^*)$. Define the set
\begin{multline*}
    \mathcal{D} (f^k, x^*) \triangleq \{ (x, g^k_{x^*} ) \in \mathbb{R}^n \times \partial f^k(x^*): \\
    \langle g^k_{x^*} , u(x^*, x) \rangle < 0 \}.
\end{multline*}
First, using the fact that $\sigma$ and $\Vert x^*-x_u^* \Vert$ are positive, we have
\begin{align*}
    \frac{ \langle g^k_{x^*} , u(x^*, x_u^*) \rangle }{\Vert x^*-x_u^* \Vert} \leq - \sigma \quad \Rightarrow \quad \langle g^k_{x^*} , u(x^*, x_u^*) \rangle < 0.  
\end{align*}
This means that we can consider the pair $(x_u^*, g^k_{x^*} )$ inside the set $(\mathcal{A} \times \mathbb{R}^n ) \cap \mathcal{D}$ instead of $\mathcal{A} \times \partial f^k(x^*) $. Next, let 
$$\mathcal{E}( \mathcal{A}, x^* ) \triangleq \{ x \in \mathcal{A} : \exists \theta \in (0,1), \;\; (1- \theta)x + \theta x^* \in \mathcal{A} \}. $$
Suppose $(x_u^{(1)}, g^k_{x^*}) \in \mathcal{D} \cap (\mathcal{E} \times \mathbb{R}^n)$. We choose $\bar{\theta}$ so that $\bar{\theta} \in (0,1)$ and $x_u^{(2)} = (1- \bar{\theta})x_u^{(1)} + \bar{\theta} x^* \in \mathcal{A}$, i.e., $x_u^{(2)}$ is in between $x_u^{(1)}$ and $x^*$, and also in the set $\mathcal{A}$. 
We have 
$$ \langle g^k_{x^*} , u(x^*, x_u^{(1)}) \rangle = \langle g^k_{x^*} , u(x^*, x_u^{(2)})  \rangle < 0$$
and so 
\begin{align*}
     \frac{ \langle g^k_{x^*} , u(x^*, x_u^{(2)}) \rangle }{\Vert x^*-x_u^{(2)} \Vert} < \frac{ \langle g^k_{x^*} , u(x^*, x_u^{(1)}) \rangle }{\Vert x^*-x_u^{(1)} \Vert},
\end{align*}
i.e., if $x_u^{(1)}$ satisfies \eqref{eqn: checkpt}, then so does $x_u^{(2)}$.
This means that we can consider the pair $(x_u^*, g^k_{x^*} )$ inside the set $\big( (\mathcal{A} \setminus \mathcal{E}) \times \mathbb{R}^n \big) \cap \mathcal{D}$ instead of $\mathcal{A} \times \partial f^k (x^*)$. However, we will show that in fact the set 
\begin{align*}
    \mathcal{A} \setminus \mathcal{E} = \big\{ x \in \mathcal{A} : \; (1- \theta)x + \theta x^* \notin \mathcal{A}, \quad \forall \theta \in (0,1) \big\}
\end{align*}
is contained in $\partial \mathcal{A}$, i.e., $\mathcal{A} \setminus \mathcal{E} \subseteq \partial \mathcal{A}$.
Suppose $x \in \mathcal{A}^\circ$ so there exists $\epsilon > 0$ such that $\mathcal{B}(x, \epsilon) \subseteq \mathcal{A}$. 
By choosing $\hat{\theta} = \frac{\epsilon}{2 \Vert x^* - x \Vert}$, we get $(1- \hat{\theta})x + \hat{\theta} x^* \in \mathcal{A}$ and $\hat{\theta} \in (0,1)$ since $x^* \notin \mathcal{A}$. 
This implies that $x \notin \mathcal{A} \setminus \mathcal{E}$ and therefore $\mathcal{A} \setminus \mathcal{E} \subseteq \partial \mathcal{A}$. Using the definition of $\tilde{\mathcal{A}}$ in \eqref{set: A_tilde}, we can then rewrite the set $\mathcal{A} \setminus \mathcal{E}$ as follows:
\begin{align*}
    \mathcal{A} \setminus \mathcal{E}( \mathcal{A}, x^* )  = \tilde{\mathcal{A}}(\mathcal{A}, x^*).
\end{align*}
From the definition of $\mathcal{X}$ in \eqref{set: X}, we have
\begin{align*}
    \big( \tilde{\mathcal{A}}( \mathcal{A}, x^* ) \times \mathbb{R}^n \big) \cap \mathcal{D} (f^k, x^*) = \mathcal{X}(f^k, \mathcal{A}, x^*).
\end{align*}
Thus, the necessary condition \eqref{eqn: nec_1} reduces to
\begin{align*}
    \min_{ ( x_u^*, g^k_{x^*} ) \in \mathcal{X} (f^k, \mathcal{A}, x^*) } \frac{ \langle g^k_{x^*} , u(x^*, x_u^*) \rangle }{\Vert x^*-x_u^* \Vert} \leq - \sigma .
\end{align*}
\end{proof}

We can interpret the necessary condition in Theorem \ref{thm: nc_general} as follows. To check whether $x^* \in \mathbb{R}^n$ can be a minimizer of $f(x)$, we can follow the inequality \eqref{eqn: result0} and search for a pair $(x_u^*, g^k_{x^*})$ with $x_u^* \in \mathcal{A}$ and $g^k_{x^*} \in \partial f^k(x^*)$ such that the pair satisfies the inequality
\begin{align}
    \frac{ \langle g^k_{x^*} , u(x^*, x_u^*) \rangle }{\Vert x^*-x_u^* \Vert} \leq - \sigma . \label{eqn: checkpt}
\end{align}
However, the inequality \eqref{eqn: result1} with $\mathcal{X} (f^k, \mathcal{A}, x^*)$ defined in \eqref{set: X} suggests that we do not have to search throughout the space $\mathcal{A} \times \partial f^k(x^*)$. Instead, we can restrict our attention to be in the set $\mathcal{X}$. 
Now we have the variables $x_u^*$ and $g^k_{x^*}$ that are coupled through the inequality $\langle g^k_{x^*} , u(x^*, x_u^*) \rangle < 0$. 
That is, if we first choose $g^k_{x^*} \in \partial f^k(x^*)$, then we can consider $x_u^*$ that is in the set $\{ x \in \partial \mathcal{A} : \langle g^k_{x^*} , u(x^*, x) \rangle < 0, \quad (1-\theta)x + \theta x^* \notin \mathcal{A}, \quad \forall \theta \in (0,1) \}$. 
Similarly, if we first choose $x_u^* \in \{ x \in \partial \mathcal{A} : \; (1-\theta)x + \theta x^* \notin \mathcal{A}, \quad \forall \theta \in (0,1) \}$, then we can consider $g^k_{x^*}$ that is in the set $\{ g \in \partial f^k(x^*) : \langle g , u(x^*, x) \rangle < 0 \}$.

If the function $f^k$ is differentiable at $x^*$, we have a single element in the set $\partial f^k(x^*)$, namely $\nabla f^k(x^*)$, and we can search for $x_u^* \in \partial \mathcal{A}$ such that $\langle \nabla f^k(x^*), u(x^*, x_u^*) \rangle < 0$. However, if the set $\mathcal{A}$ is arbitrary, this search may be computationally expensive. In the next section, we consider additional structure on the set $\mathcal{A}$ to simplify the search.

\begin{remark}
Note that the set $\mathcal{A}^\circ \subseteq \mathcal{M}( f^k, \mathcal{A}, \sigma )$.  To see this, note that for all $x^* \in \mathcal{A}^\circ$, there exists $\epsilon > 0$ such that $\mathcal{B} ( x^*, \epsilon ) \subset \mathcal{A}^\circ$. Suppose $g \in \partial f^k(x^*)$. We can choose $f^u(x) = \frac{\sigma_u}{2} \big\| x - \big( x^* + \frac{g}{\sigma_u} \big) \big\|^2$ where $\sigma_u = \frac{2 k \|g \|}{\epsilon}$ and $k = \max \big\{ 1, \frac{\sigma \epsilon}{2 \|g \|} \big\} $. One can verify that $x^*_u \in \mathcal{B} ( x^*, \epsilon )$, $\nabla f^u(x^*) = -g$, and $\sigma_u \geq \sigma$.  
\end{remark}

\section{Analysis for the Case where $\mathcal{A}$ is a Ball}

Here, we consider additional structure on the uncertainty set $\mathcal{A}$ in order to provide a more specific characterization of the region $\mathcal{M}$. In particular, we consider $\mathcal{A} = \Bar{\mathcal{B}} (\bar{x}, \epsilon_0)$, where $\bar{x}$ is the best guess of what the true parameter $x_u^*$ is, and $\epsilon_0$ is the maximum possible deviation of the true minimizer from our best guess. 

We begin by investigating a property of the necessary condition \eqref{eqn: result1} under a coordinate transformation. Suppose $x = (x_{(1)}, x_{(2)}, \ldots, x_{(n)}) \in \mathbb{R}^n$ and $x^* \notin \overline{\mathcal{B}} (\bar{x}, \epsilon_0)$. 
Let $\mathbf{T}$ and $\mathbf{R}$ be the translation and rotation operators such that $\mathbf{R} ( \mathbf{T}(\bar{x}) ) = \mathbf{0}$, $\mathbf{R} ( \mathbf{T}(x^*) ) = (\tilde{x}_{(1)}^*, 0, \ldots, 0)$ with $\tilde{x}_{(1)}^* > 0$, and $\mathbf{R} (g^k)  = (\tilde{g}_{(1)}, \tilde{g}_{(2)}, 0, \ldots, 0)$ with $\tilde{g}_{(2)} \geq 0$ while preserving the distance between any two points. In other words, given the ball $\bar{\mathcal{B}} (\bar{x}, \epsilon_0)$, a point $x^*$ and a vector $g^k$, we transform the coordinates so that the ball is centered at the origin, the point $x^*$ lies on the $x$-axis, and the vector $g^k$ lies on the $x$-$y$ plane.

Next, consider the expression
$\frac{ \langle g , u(x^*, x_u^*) \rangle }{\Vert x^*-x_u^* \Vert}$.
Notice that both numerator and denominator can be written as inner products.
Since  $\mathbf{R}$ is a unitary operator, we have  
\begin{align*}
    \frac{  \big\langle \mathbf{R}( g ) , u\big( \mathbf{R} (\mathbf{T} (x^*)), \mathbf{R}( \mathbf{T}( x_u^*))\big) \big\rangle }{\Vert \mathbf{R}( \mathbf{T}( x^*)) - \mathbf{R}( \mathbf{T}( x_u^*)) \Vert}  
    = \frac{ \langle g , u(x^*, x_u^*) \rangle }{\Vert x^*-x_u^* \Vert}.
\end{align*}
This means that even though we use the coordinate transformation $\mathbf{R} ( \mathbf{T} (\cdot) )$, we can still apply Theorem \ref{thm: nc_general}. Therefore, for the purpose of deriving our main result, without loss of generality, we can consider $\bar{x} = \mathbf{0}$, $x^* = (x_{(1)}^*, 0, \ldots, 0)$ where $x_{(1)}^* > \epsilon_0$, and $g \;(= g^k) = (g_{(1)}, g_{(2)}, 0, \ldots, 0)$, where $g_{(2)} \geq 0$.   

Before going into the result, we introduce some definitions that will appear in the theorem. For any given $x^* \in \mathbb{R}^n$, define $z_1 (x^*) \in \mathbb{R}^n$ as
\begin{align}
    z_1(x^*) \triangleq \argmin_{x \in \overline{\mathcal{B}} (\bar{x}, \epsilon_0) } \Vert x - x^* \Vert . \label{def: z_1}
\end{align}
By our assumption that $x^* = (x_{(1)}^*, 0, \ldots, 0)$, we have $z_1(x^*) = (\epsilon_0, 0, \ldots, 0)$.
Since $x^* \notin \overline{\mathcal{B}} (\bar{x}, \epsilon_0)$, the point $z_1$ is unique and is on $\partial \overline{\mathcal{B}} (\bar{x}, \epsilon_0)$.
If $g \neq \alpha (x^* - \bar{x})  = (\alpha x_{(1)}^*, 0, \ldots, 0)$ for all $\alpha \geq 0$ (i.e., $\angle (g, x^* - \bar{x}) \neq 0$), we define the set $\mathcal{P}$ to be such that
\begin{align*}
    \mathcal{P}( g, x^* ) \triangleq \argmin_{x \in \partial \overline{\mathcal{B}} (\bar{x}, \epsilon_0) } \angle ( g, x - x^* ),
\end{align*}
the point $z_2 \in \mathbb{R}^n$ to be such that
\begin{align}
    z_2( g, x^* ) \triangleq \argmin_{x \in \mathcal{P} (g, x^*)} \Vert x - x^* \Vert, \label{def: z_2}
\end{align}
and the curve $\mathcal{C}_0 (\bar{x}, \epsilon_0, g, x^*)$ to be the shortest path on the surface $\partial \overline{\mathcal{B}} (\bar{x}, \epsilon_0)$ that connects $z_1$ and $z_2$ together, i.e., $\mathcal{C}_0$ is the geodesic path between $z_1$ and $z_2$ on $\partial \overline{\mathcal{B}} (\bar{x}, \epsilon_0)$. 

To clarify these definitions, we introduce two more objects. Let $L$ be the ray that starts from the point $x^*$ and runs parallel to the vector $g$ i.e.,
\begin{align*}
    L(g, x^*) = \{ x \in \mathbb{R}^n: \; \exists \; t \in [0, \infty), \; x = x^* + t g \}.
\end{align*}
If $g \neq \alpha (x^* - \bar{x}) = (\alpha x_{(1)}^*, 0, \ldots, 0)$ for all $\alpha \in \mathbb{R}$, let $P_2$ be the 2-dimensional plane that contains the vectors $g$ and $x^* - \bar{x}$ as its bases, and contains the point $x^*$, i.e.,
\begin{align*}
    P_2 (\bar{x}, g, x^*) &\triangleq \{ x \in \mathbb{R}^n: \; \exists \; s, t \in \mathbb{R} \;\; \text{such that} \\
    & \qquad \qquad \qquad \qquad x = x^* + sg + t(x^* - \bar{x}) \} \\
    &= \{ x \in \mathbb{R}^n: \; x_{(3)} = x_{(4)} = \ldots = x_{(n)} = 0 \},
\end{align*}
where the second equality follows from the fact that $x^* = (x^*_{(1)}, 0, 0, \ldots, 0)$ and $g = (g_{(1)}, g_{(2)}, 0, \ldots, 0)$.

There are two possible cases: (i) the ray $L$ passes through the ball $\overline{\mathcal{B}} (\bar{x}, \epsilon_0)$ and (ii) the ray $L$ does not pass through the ball $\overline{\mathcal{B}} (\bar{x}, \epsilon_0)$. 

In the first case, we have
\begin{align*}
    \min_{x \in \partial \overline{\mathcal{B}} (\bar{x}, \epsilon_0) } \angle ( g, x - x^* ) = 0,
\end{align*}
and there are either one or two elements in the set $\mathcal{P}$. The point $z_2$ is the one that closer to the point $x^*$. Note that $z_2 \in P_2$. The illustration of the first case is shown in Fig. \ref{fig: cir_case1}.

In the second case, we have
\begin{align*}
    \min_{x \in \partial \overline{\mathcal{B}} (\bar{x}, \epsilon_0) } \angle ( g, x - x^* ) > 0.
\end{align*}
The vector $z_2 - x^*$ is a tangent vector at the point $z_2$ on the ball $\bar{\mathcal{B}}$ and has angle $\angle (z_2 - x^*, \bar{x} - x^*) = \arcsin \big(\frac{ \epsilon_0}{ \| x^* - \bar{x} \| } \big)$. Furthermore, the point $z_2$ is on the plane $P_2$ since $z_2 - x^*$ and $\bar{x} - x^*$ must be on the same 2D-plane in order to minimize the angle between them. The illustration of the second case is shown in Fig. \ref{fig: cir_case2}.

Since $P_2$ passes through the center $\bar{x}$ of the ball $\bar{\mathcal{B}} (\bar{x}, \epsilon_0)$, we can define the great circle $\mathcal{G} \subset P_2$ which is the intersection of $\partial \bar{\mathcal{B}}$ with $P_2$. Since $z_1$ and $z_2$ are in $\mathcal{G}$ (and also in $P_2$), the geodesic path $\mathcal{C}_0$ is in $P_2$. The geodesic path in both cases is also shown in Fig. \ref{fig: cir_case1} and Fig. \ref{fig: cir_case2}.

Before stating the theorem, define the open half-space
\begin{align*}
    \mathcal{H}(g, x^*) \triangleq \{ x \in \mathbb{R}^n: \langle g , u(x^*, x) \rangle < 0 \},
\end{align*}
Note that $\mathcal{C}_0 (\bar{x}, \epsilon_0, g, x^*) \cap \mathcal{H}(g, x^*) \neq \emptyset$ as long as $\angle (g, z_2 - x^*) < \frac{\pi}{2}$ or equivalently, $\angle (g, \bar{x} - x^*) < \frac{\pi}{2} + \arcsin \big(\frac{ \epsilon_0}{ \| x^* - \bar{x} \| } \big)$ as shown in Fig. \ref{fig: grad_case1} and \ref{fig: grad_case2}.

We now come to the main result of this section.

\begin{figure}
  \centering
  \includegraphics[width=0.8\linewidth]{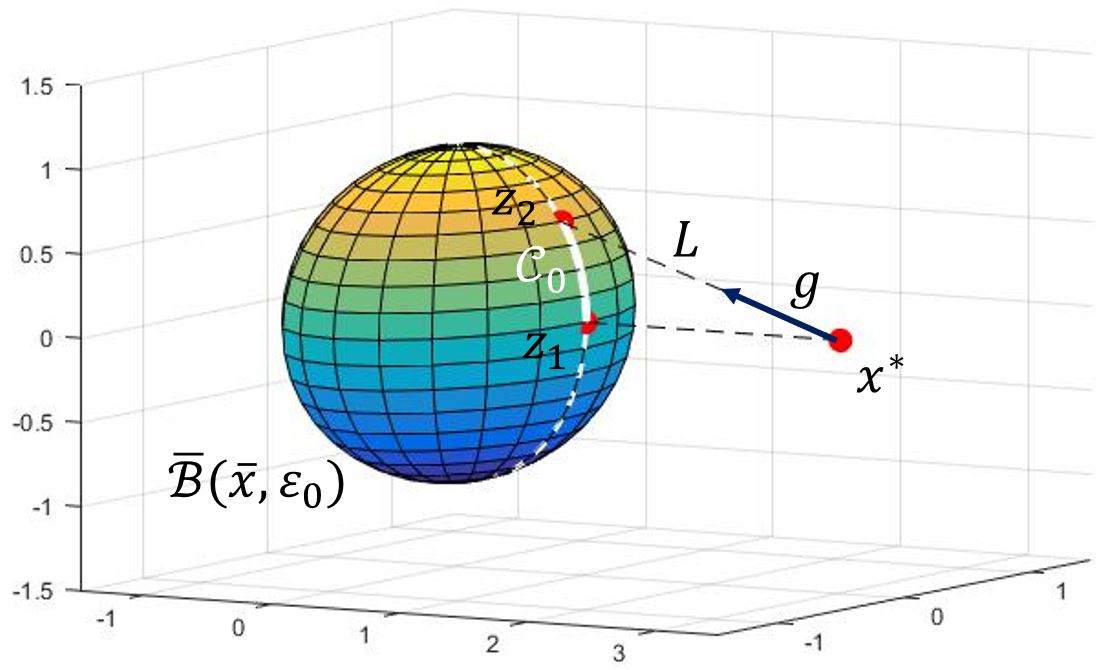}
     \caption{The points $z_1$ and $z_2$, and the curve $\mathcal{C}_0$ on the surface $\partial \overline{\mathcal{B}} (\bar{x}, \epsilon_0)$ in the case that the ray $L$ passes through the ball $\overline{\mathcal{B}} (\bar{x}, \epsilon_0)$.}
  \label{fig: cir_case1}
\end{figure}  

\begin{figure}
  \centering
    \includegraphics[width=0.8\linewidth]{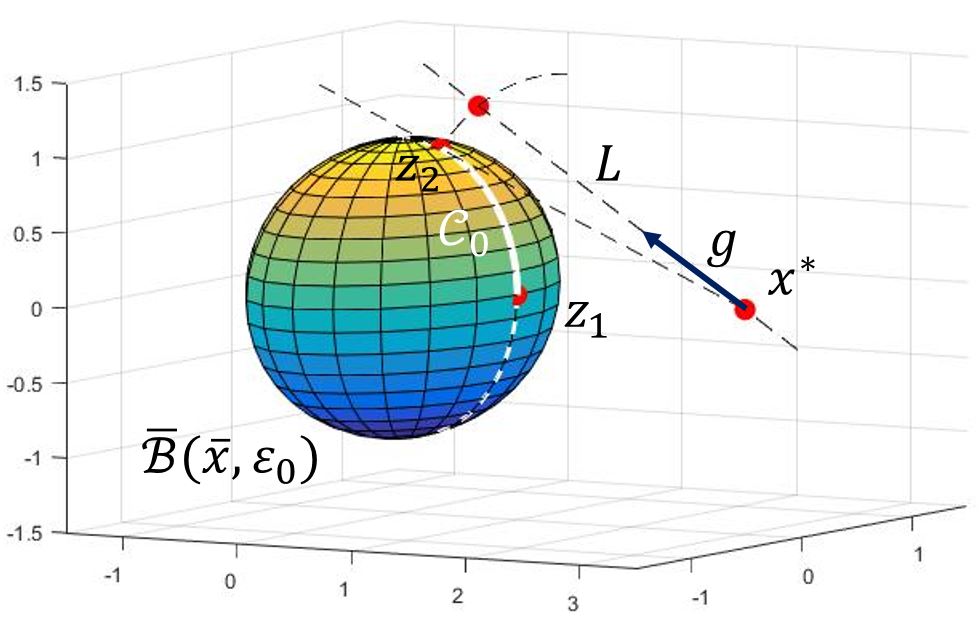}
  \caption{The points $z_1$ and $z_2$, and the curve $\mathcal{C}_0$ on the surface $\partial \overline{\mathcal{B}} (\bar{x}, \epsilon_0)$ in the case that the ray $L$ does not pass through the ball $\overline{\mathcal{B}} (\bar{x}, \epsilon_0)$.}
  \label{fig: cir_case2}
\end{figure}

\begin{theorem}
  Suppose $f^k \in \mathcal{F}$ and $\epsilon_0 > 0$. A necessary condition for a point $x^* \in \mathbb{R}^n$ to be in $\mathcal{M} \big(f^k, \overline{\mathcal{B}} (\bar{x}, \epsilon_0), \sigma \big) \setminus \overline{\mathcal{B}} (\bar{x}, \epsilon_0)$ is 
\begin{align}
    \min_{ ( x_u^*, g^k ) \in \tilde{\mathcal{X}} (f^k, \bar{x}, \epsilon_0, x^*) } \frac{ \langle g^k , u(x^*, x_u^*) \rangle }{\Vert x^*-x_u^* \Vert} \leq - \sigma \label{eqn: result}
\end{align}
where 
\begin{align}
    \tilde{\mathcal{X}} (f^k, \bar{x}, \epsilon_0, x^*) \triangleq \big\{ (x, g) \in \mathcal{C}_0 (\bar{x}, \epsilon_0, g, x^*) \times \partial f^k (x^*) \}. \label{set: X_tilde}
\end{align} \label{thm: sphere}
\end{theorem}

\begin{proof}
For a given $g^k_{x^*} \in \partial f^k (x^*)$ with $g^k_{x^*} \neq \mathbf{0}$, we consider the angle $\angle (g^k_{x^*}, x^* - \bar{x})$ in two disjoint cases: 
\begin{enumerate}[(a)]
    \item Suppose the gradient $g^k_{x^*}$ is colinear with the vector $x^* - \bar{x}$.
\begin{enumerate}[(i)]
    \item If $g^k_{x^*} = \alpha (x^* - \bar{x})$ for some $\alpha > 0$ (i.e., $g^k_{x^*}$ is pointing directly away from $\bar{\mathcal{B}} (\bar{x}, \epsilon_0)$ on the $x_{(1)}$-axis), then $\langle g^k_{x^*} , u(x^*, x) \rangle > 0$ for all $x \in \overline{\mathcal{B}} (\bar{x}, \epsilon_0)$. Thus, no points in $\overline{\mathcal{B}} (\bar{x}, \epsilon_0)$ can satisfy the inequality \eqref{eqn: checkpt}.
    \item If $g^k_{x^*} = \alpha (x^* - \bar{x})$ for some $\alpha < 0$ (i.e., $g^k_{x^*}$ is pointing directly toward $\bar{\mathcal{B}} (\bar{x}, \epsilon_0)$ on the $x_{(1)}$-axis), then the ray $L$ passes through the ball $\bar{\mathcal{B}} (\bar{x}, \epsilon_0)$ at $z_1$, and thus $\{ z_1 (x^*) \} = \{ z_2 (g^k_{x^*}, x^*) \} = \mathcal{C}_0$. Furthermore, $\bar{\mathcal{B}} \subset \mathcal{H}$. For simplicity of notation, we will omit the arguments and write $z_1 (x^*)$ and $z_2 (g^k_{x^*}, x^*)$ as $z_1$ and $z_2$, respectively.
    From \eqref{def: z_2}, for all $x \in \partial \bar{\mathcal{B}}$, we have 
    \begin{align}
        &\angle (g^k_{x^*}, u(z_2, x^*)) \leq \angle (g^k_{x^*}, u(x, x^*))  \nonumber \\
        \Rightarrow \;\; &\angle (g^k_{x^*}, u(x^*, z_2)) \geq \angle (g^k_{x^*}, u(x^*, x)) \nonumber \\
        \Rightarrow \;\; &\cos \angle (g^k_{x^*}, u(x^*, z_2)) \leq \cos \angle (g^k_{x^*}, u(x^*, x)) \nonumber \\
        \Rightarrow \;\; &\langle g^k_{x^*} , u(x^*, z_2) \rangle \leq \langle g^k_{x^*} , u(x^*, x) \rangle. \label{eqn: inner}
    \end{align}
    Since $z_2$, $x \in \mathcal{H}$, we have $\langle g^k_{x^*} , u(x^*, z_2) \rangle \leq \langle g^k_{x^*} , u(x^*, x) \rangle < 0$.
    In addition, from \eqref{def: z_1}, for all $x \in \partial \bar{\mathcal{B}}$, we have $0 < \| x^* - z_1 \| \leq \| x^* - x \|$. Since $z_1 = z_2$ in this case, we obtain 
    \begin{align*}
        \frac{- \| g^k_{x^*} \|}{ \| x^* - z_1 \|} = \frac{ \langle g^k_{x^*} , u(x^*, z_1) \rangle }{\Vert x^*-z_1 \Vert} \leq \frac{ \langle g^k_{x^*} , u(x^*, x) \rangle }{\Vert x^*-x \Vert}
    \end{align*}
    for all $x \in \partial \bar{\mathcal{B}}$. Thus, it suffices to only check $z_1 \in \mathcal{C}_0$ to see if \eqref{eqn: checkpt} is satisfied.
    \end{enumerate}
    \item Suppose the gradient $g^k_{x^*}$ is not colinear with the vector $x^* - \bar{x}$. Then we can define the points $z_1$ and $z_2$ as described earlier.
    If 
    $$\angle (g^k_{x^*}, \bar{x} - x^*) \geq \frac{\pi}{2} + \arcsin \Big( \frac{\epsilon_0}{\| x^* - \bar{x} \|} \Big),$$ 
    then $\bar{\mathcal{B}} (\bar{x}, \epsilon_0) \cap \mathcal{H}(g^k_{x^*}, x^*) = \emptyset$ as shown in Fig. \ref{fig: grad_case1}, and no points in $\overline{\mathcal{B}} (\bar{x}, \epsilon_0)$ can satisfy the inequality \eqref{eqn: checkpt}.
    If 
    $$\angle (g^k_{x^*}, \bar{x} - x^*) < \frac{\pi}{2} + \arcsin \Big( \frac{\epsilon_0}{\| x^* - \bar{x} \|} \Big),$$ 
    then $\bar{\mathcal{B}} (\bar{x}, \epsilon_0) \cap \mathcal{H}(g^k_{x^*}, x^*) \neq \emptyset$ and $z_2 \in \mathcal{H}(g^k_{x^*}, x^*)$  as shown in Fig. \ref{fig: grad_case2}.
    In this case, consider a point $x \in \partial \bar{\mathcal{B}} (\bar{x}, \epsilon_0) \cap \mathcal{H} (g^k_{x^*}, x^*)$ and $x \notin \mathcal{C}_0$. 
    \begin{enumerate}[(i)]
    \item Suppose $\Vert x - x^* \Vert > \Vert z_2 - x^* \Vert$. By the definition of $z_2$ in \eqref{def: z_2}, we have $\angle (g^k_{x^*}, u(z_2, x^*)) \leq \angle (g^k_{x^*}, u(x, x^*))$. Since $z_2$, $x \in \mathcal{H}$, using the same argument as \eqref{eqn: inner}, we get $\langle g^k_{x^*} , u(x^*, z_2) \rangle \leq \langle g^k_{x^*} , u(x^*, x) \rangle < 0$. Therefore, 
    \begin{align*}
    \frac{ \langle g^k_{x^*} , u(x^*, z_2) \rangle }{\Vert z_2-x^* \Vert} < \frac{ \langle g^k_{x^*} , u(x^*, x) \rangle }{\Vert x-x^* \Vert},
    \end{align*}
    i.e., if $x$ satisfies \eqref{eqn: checkpt}, then so does $z_2$. This means that we can consider $z_2 \in \mathcal{C}_0$ instead of any point in $\partial \bar{\mathcal{B}} (\bar{x}, \epsilon_0) \cap \mathcal{H} (g^k_{x^*}, x^*)$ with greater distance from $x^*$.
    \item Suppose $\Vert x - x^* \Vert \leq \Vert z_2 - x^* \Vert$. Since $\mathcal{C}_0$ is connected and $h(y) = \| y - x^* \|$ is a continuous function, $\{ h( y) : y \in \mathcal{C}_0 \}$ is connected. Then, we have 
    \begin{align*}
        \Big[ \| z_1 - x^* \|, \|z_2 - x^*\| \Big] \subseteq \big\{ \| y - x^* \| : y \in \mathcal{C}_0 \big\}.
    \end{align*}
    Thus, there exists a $z \in \mathcal{C}_0 \cap \mathcal{H}$ such that $\Vert z - x^* \Vert = \Vert x - x^* \Vert$. However, since $\mathcal{C}_0 \subset P_2$, we get that 
    \begin{align*}
        \angle ( g^k_{x^*}, u(z, x^*) ) \leq \angle ( g^k_{x^*}, u(x, x^*) ).
    \end{align*}
    Furthermore, since $z$, $x \in \mathcal{H}$, using the same argument as \eqref{eqn: inner}, we get $\langle g^k_{x^*} , u(x^*, z) \rangle < \langle g^k_{x^*} , u(x^*, x) \rangle < 0$. In this case, we also have 
    \begin{align*}
        \frac{ \langle g^k_{x^*} , u(x^*, z) \rangle }{\Vert z-x^* \Vert} \leq \frac{ \langle g^k_{x^*} , u(x^*, x) \rangle }{\Vert x-x^* \Vert},
    \end{align*}
    i.e., if $x$ satisfies \eqref{eqn: checkpt}, then so does $z$. 
    \end{enumerate}
\end{enumerate}
Thus, we conclude that for each point $x \in \partial \bar{\mathcal{B}} \cap \mathcal{H}$, there is a point $z \in \mathcal{C}_0$ such that $\frac{\langle g^k_{x^*}, u(x^*, z) \rangle}{ \| x^* - z \| } \leq \frac{ \langle g^k_{x^*}, u(x^*, x) \rangle }{ \| x^* - x \| }$. Therefore, to check if there is a point $x \in \partial \bar{\mathcal{B}} \cap \mathcal{H}$ satisfying \eqref{eqn: checkpt}, we only need to check points in $\mathcal{C}_0$,  yielding \eqref{eqn: result}. 
\end{proof}

\begin{figure}
  \centering
    \includegraphics[width=0.65 \linewidth]{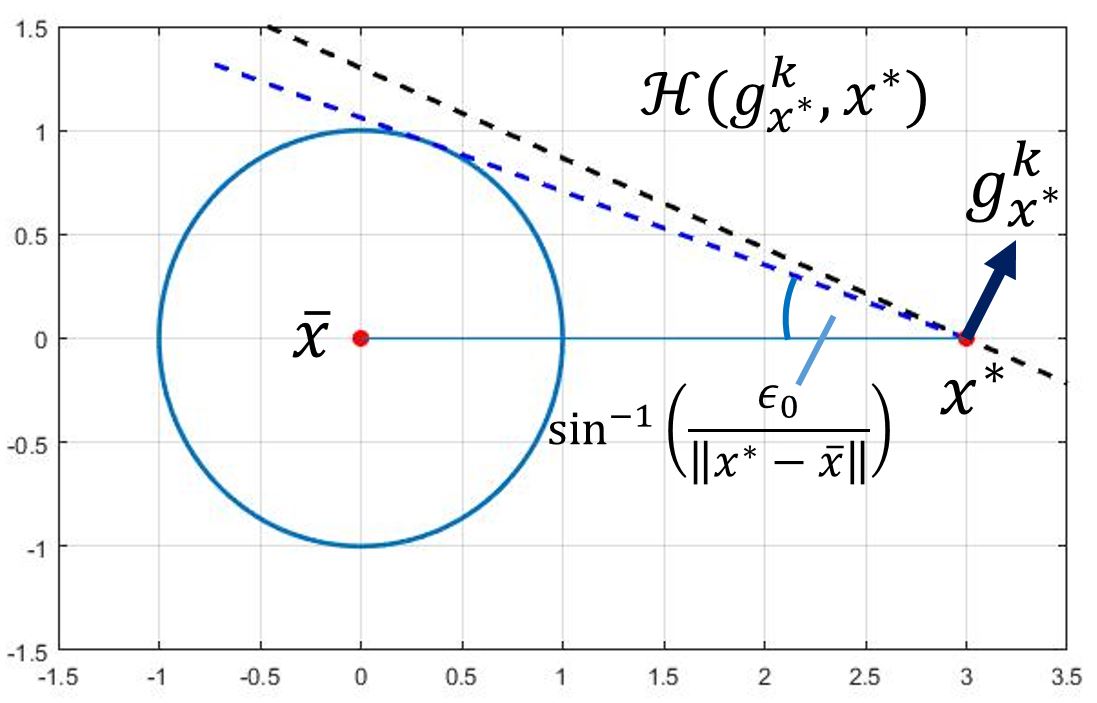}
     \caption{The area above the black dotted line is $\mathcal{H}(g^k_{x^*}, x^*)$ and the blue dotted line shows the angle $\arcsin \big( \frac{\epsilon_0}{\| x^* - \bar{x} \|} \big)$. In this case, the angle $\angle (g^k_{x^*}, \bar{x} - x^*) \geq \frac{\pi}{2} + \arcsin \big( \frac{\epsilon_0}{\| x^* - \bar{x} \|} \big)$, so $\bar{\mathcal{B}} (\bar{x}, \epsilon_0) \cap \mathcal{H}(g^k_{x^*}, x^*) = \emptyset$.}
     \label{fig: grad_case1}
\end{figure}

\begin{figure}
    \centering
    \includegraphics[width=0.65\linewidth]{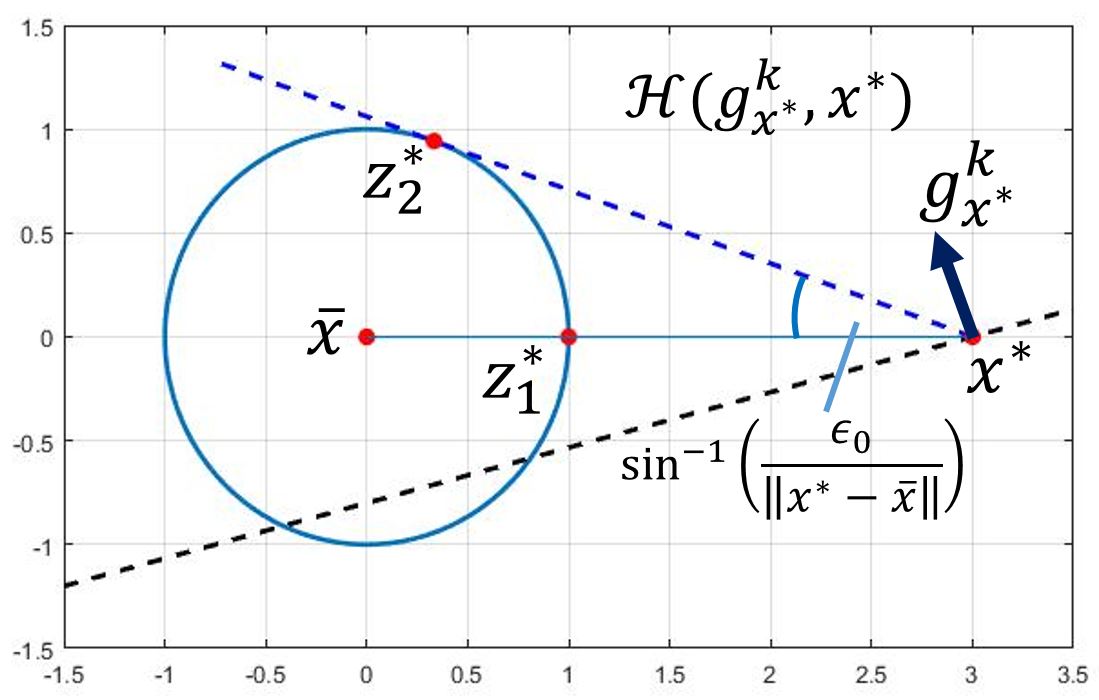}
     \caption{The area above the black dotted line is $\mathcal{H}(g^k_{x^*}, x^*)$ and the blue dotted line shows the angle $\arcsin \big( \frac{\epsilon_0}{\| x^* - \bar{x} \|} \big)$. In this case, the angle $\angle (g^k_{x^*}, \bar{x} - x^*) < \frac{\pi}{2} + \arcsin \big( \frac{\epsilon_0}{\| x^* - \bar{x} \|} \big)$, so $\bar{\mathcal{B}} (\bar{x}, \epsilon_0) \cap \mathcal{H}(g^k_{x^*}, x^*) \neq \emptyset$.}
    \label{fig: grad_case2}
\end{figure}

In fact, we can replace $\mathcal{C}_0 (\bar{x}, \epsilon_0, g, x^*)$ in Theorem \ref{thm: sphere} by $\mathcal{C}_0 (\bar{x}, \epsilon_0, g, x^*) \cap \mathcal{H}(g, x^*)$. However, for  simplicity of exposition, we forego the discussion of this further reduction in search space.

The set $\tilde{\mathcal{X}} (f^k, \bar{x}, \epsilon_0, x^*)$ defined in \eqref{set: X_tilde} suggests that we do not have to search for a pair $(x^*, g^k)$ that satisfies the inequality
\begin{align*}
    \frac{ \langle g^k , u(x^*, x_u^*) \rangle }{\Vert x^*-x_u^* \Vert} \leq - \sigma 
\end{align*}
throughout the set $\mathcal{X} (f^k, \mathcal{A}, x^*)$ defined in \eqref{set: X} but can instead restrict our attention to be in the set $\tilde{\mathcal{X}} (f^k, \bar{x}, \epsilon_0, x^*)$ in \eqref{set: X_tilde}. Since the curve $\mathcal{C}_0$ depends on the vector $g^k$ that we choose from $\partial f^k(x^*)$, we have to first select $g^k \in \partial f^k(x^*)$ and then we can consider the points on the curve $\mathcal{C}_0$ to see if they satisfy \eqref{eqn: checkpt}.
We will use this in the algorithm for computing the region $\mathcal{M}$ in the next section.

\section{Algorithm and Example}

\subsection{Algorithm}

Consider the case from the previous section where the uncertainty set is a ball, i.e., $\mathcal{A} = \bar{\mathcal{B}} (\bar{x} , \epsilon_0)$. In this subsection, we will give an algorithm (Algorithm \ref{alg: ball}) to identify the region that satisfies the necessary condition \eqref{eqn: result}. We provide a discussion of each of the steps below. 

\begin{algorithm}
\caption{Region $\mathcal{M}$ Identification (Ball Case)} \label{alg: ball}
Let $X \subseteq \dom f^k$ be a set of points in the space \\
\textbf{Input} $X$, $f^k \in \mathcal{F}$, $\bar{x} \in \mathbb{R}^n$, $\epsilon_0 > 0$, and $\sigma > 0$  \\
\textbf{Output} $minimizer(X)$
\begin{algorithmic}[1]
\For {$x^* \in X$} \Comment{Loop through the space} 
\State $minimizer(x^*) \gets \text{false}$ 
\State $d \gets \| \bar{x} - x^* \|$
\State $g \gets \nabla f^k(x^*)$ 
\State $\alpha \gets \angle (g, \bar{x} - x^* )$
\If {$\alpha < \frac{\pi}{2} + \arcsin \big( \frac{\epsilon_0}{  d } \big) $}   
\For{$\theta \in \big[0, \arccos (\frac{\epsilon_0}{ d }) \big]$}
\State $\| x^* - x_u^* \| \gets \sqrt{ d^2 + \epsilon_0^2 - 2 \epsilon_0 d \cos \theta}$
\State $\angle (g, x^* - x_u^*) \gets \alpha + \Big[ \pi - \arcsin \big( \frac{ \epsilon_0 \sin \theta }{ \| x^* - x_u^* \| } \big) \Big]$
\State $\langle g , u(x^*, x_u^*) \rangle \gets \| g \| \cos  \angle (g, x^* - x_u^*) $ 
\If {$\frac{ \langle g , u(x^*, x_u^*) \rangle }{\Vert x^*-x_u^* \Vert} \leq -\sigma$}
\State $minimizer(x^*) \gets \text{true}$
\EndIf
\EndFor
\EndIf
\EndFor
\State \Return $minimizer(X)$
\end{algorithmic}
\end{algorithm}

Let $X$ be a set of points; we wish to check whether each point in $X$ is a potential minimizer of $f^k + f^u$. For simplicity, we assume that the function $f^k$ is differentiable, i.e., $\partial f^k(x^*) = \{ \nabla  f^k (x^*) \}$ and the set of points $X \subseteq \dom f^k$.  
For example, we can use \texttt{linspace} in MATLAB to form a range for each axis, followed by using \texttt{meshgrid} to construct $X$. The object $minimizer$ is an array that keeps a Boolean value for each point in $X$ to indicate whether it is a potential minimizer. 
First, we loop through each point $x^*$ in the set $X$ and assign Boolean `false' to that $x^*$. In order to change the Boolean to be `true', the point $x^*$ has to satisfy the inequality \eqref{eqn: result}. 
Before checking that inequality, we need to compute several intermediate variables. In the algorithm, we compute the distance between the center of the ball $\bar{x}$ and the point $x^*$ ($d \gets \| \bar{x} - x^* \|$), the gradient of $f^k$ at $x^*$ ($g \gets \nabla f^k(x^*)$), and the angle between the gradient and reference ($\alpha \gets \angle (g, \bar{x} - x^*)$). Note that we can compute $\alpha$ explicitly by
\begin{align*}
    \alpha \gets \angle (g, \bar{x} - x^*) = \arccos \Big( \frac{ \langle g, \bar{x} - x^* \rangle }{ \|g \| \| \bar{x} - x^* \|} \Big).
\end{align*}

We then verify the condition 
\begin{align*}
    \angle (g, \bar{x} - x^*) < \frac{\pi}{2} + \arcsin \Big( \frac{\epsilon_0}{\| x^* - \bar{x} \|} \Big) 
\end{align*}
(line 6); if this is not satisfied, no points in $\overline{\mathcal{B}} (\bar{x}, \epsilon_0)$ can satisfy the inequality \eqref{eqn: checkpt} as argued in the proof of Theorem \ref{thm: sphere} and illustrated in Fig. \ref{fig: grad_case1}.
The next step is to compute the path $\mathcal{C}_0$, which we parametrize by using the variable $\theta$. The variable $\theta$ in the algorithm corresponds to 
\begin{align*}
    \theta = \angle (x_u^* - \bar{x}, x^* - \bar{x}) \quad \text{where} \quad x_u^* \in \mathcal{C}_0
\end{align*}
as shown in Fig. \ref{fig: compute_distance}.
So, we need to know the range of $\theta$ that characterizes the path $\mathcal{C}_0$. 
This range can be computed by considering the points $z_1$ and $z_2$, at which the angle $\theta$ equals $0$ and $\arccos (\frac{\epsilon_0}{ \| \bar{x} - x^* \| } )$, respectively, as shown in Fig. \ref{fig: compute_range}.
Consider Fig. \ref{fig: compute_distance}. For each $\theta$ in the range (discretized to a sufficiently fine resolution), we can compute the distance $\| x^* - x_u^* \|$ (line 8) by using the cosine law.
Consider Fig. \ref{fig: compute_angle}. We can compute the angle $\angle (g, x^* - x_u^*)$ (line 9) by using
\begin{align*}
    &\angle (g, x^* - x_u^*) = \angle (g, x^* - \bar{x}) + \arcsin \Big( \frac{ \epsilon_0 \sin \theta }{ \| x^* - x_u^* \| } \Big) \\
    &\quad \text{and} \quad \angle (g, x^* - \bar{x}) = ( \pi - \angle (g, \bar{x} - x^*) ).
\end{align*}
After that we compute the inner product $\langle g , u(x^*, x_u^*) \rangle$ (line 10).
Finally, we can compute the LHS of \eqref{eqn: result} and compare it to $- \sigma$. If the inequality \eqref{eqn: result} is satisfied by the current values $x^*$ and $\theta$, we set the Boolean associated to this $x^*$ to be `true'.

\begin{figure}
    \centering
    \includegraphics[width=0.60\linewidth]{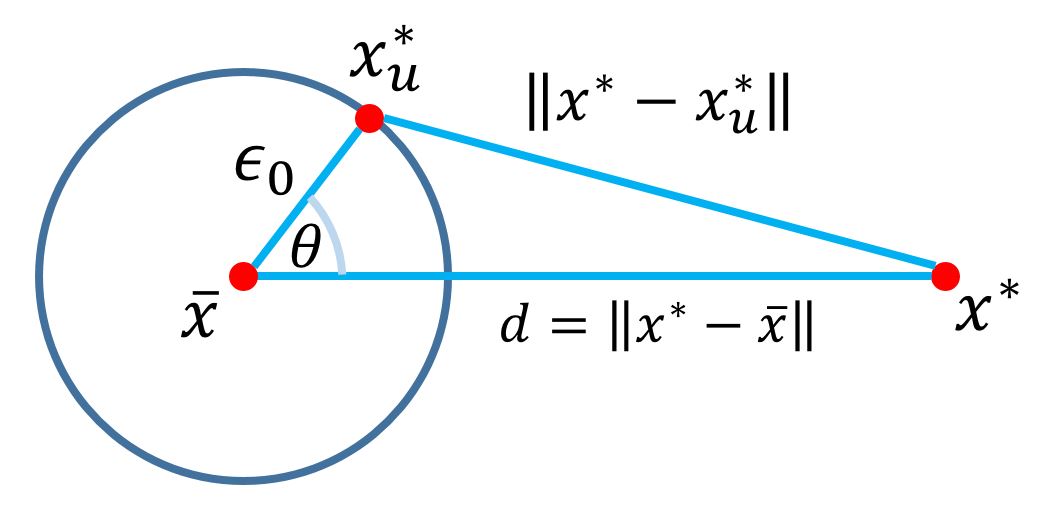}
     \caption{Given $\epsilon_0$, $d$, and $\theta$, we can compute $\| x^* - x_u^* \|$.}
    \label{fig: compute_distance}
\end{figure}  

\begin{figure}
    \centering
    \includegraphics[width=0.60\linewidth]{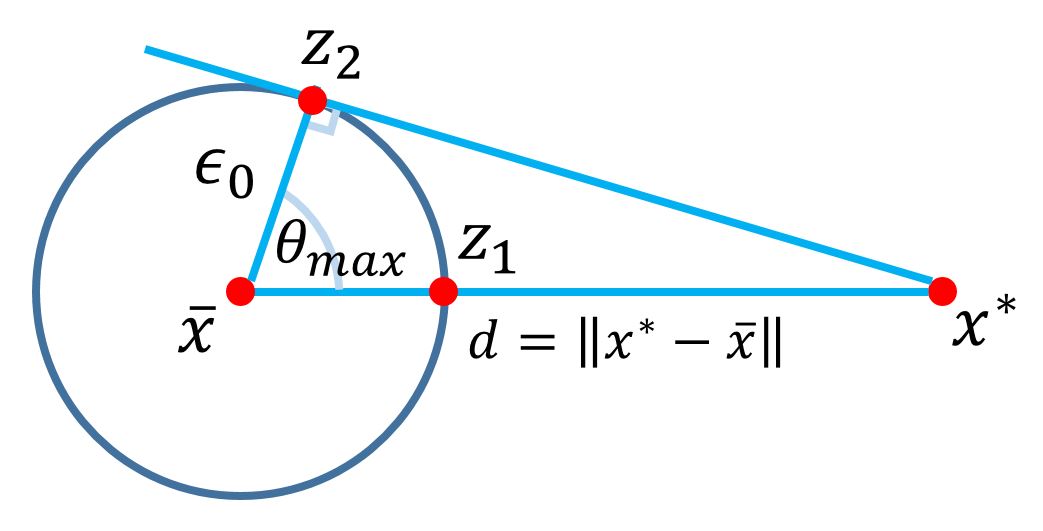}
     \caption{Given $\epsilon_0$ and $d$, we can compute $\angle (z_2 - \bar{x}, x^* - \bar{x})$.}
    \label{fig: compute_range}
\end{figure}  

\begin{figure}
    \centering
    \includegraphics[width=0.70\linewidth]{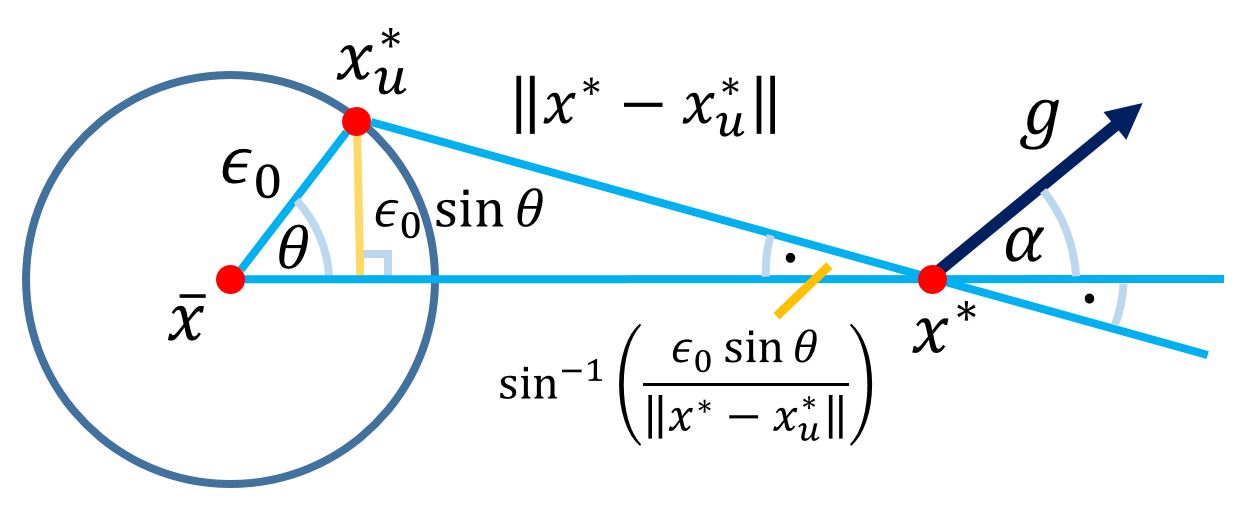}
     \caption{Given $\epsilon_0$, $\theta$, $\| x^* - x_u^* \|$, and $\alpha$, we can compute $\angle (g, x^* - x_u^*)$.}
    \label{fig: compute_angle}
\end{figure}

\subsection{Example}

Consider the known function $f^k(x) = (x_1 - 2)^2 + x_2^2$, and suppose the unknown function $f^u$ has minimizer in the ball centered at $\bar{x} = (0, 0)$. We vary the radius of the ball of uncertainty ($\epsilon_0$) among the values 0.1, 0.4, and 0.8, and the strong convexity parameter ($\sigma$) of the function $f^u$ among the values 0.25, 2.0, and 5.0. Examples of the region that contains the possible minimizer of the sum $f^k + f^u$ are shown in Fig. \ref{fig: simulation}. In the figure, the function $f^k(x)$ is shown by using level curves and the uncertainty ball is shown by the beige circle. The region containing the possible minimizers of $f^k + f^u$ (i.e., the set of points $x \in \mathbb{R}^n$ that satisfies \eqref{eqn: result}) is shown in blue (it contains the uncertainty set within it). Note that the solution region shrinks with increasing $\sigma$ and grows with increasing $\epsilon_0$.

\begin{figure}
  \centering
  \includegraphics[width=\linewidth]{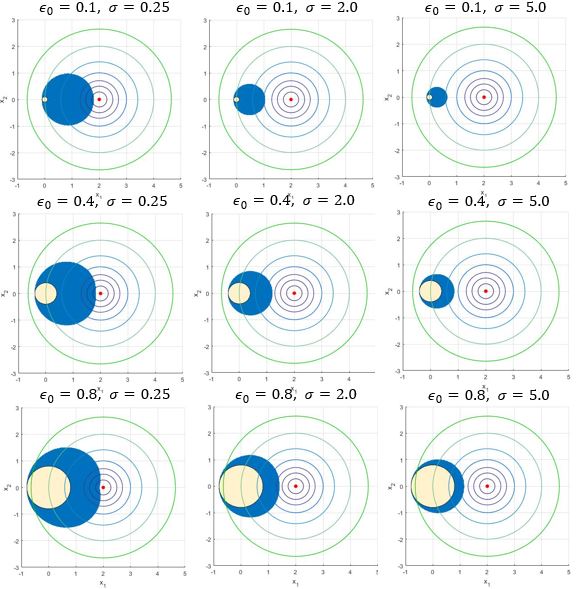}
  \caption{The function ${f}^k(x) = (x_1 - 2)^2 + x_2^2$ is shown by the level curves while the balls $ \overline{\mathcal{B}} (\bar{x}, \epsilon_0)$ with the center at $(0,0)$ are shown by the beige circle. The radius of the ball of uncertainty ($\epsilon_0$) and the strong convexity parameter ($\sigma$) of the function $f_m$ are varied and the solution sets are shown by the dark blue regions.}
  \label{fig: simulation}
\end{figure}

\section{Conclusions}

In this paper, we studied the properties of the minimizer of the sum of convex functions in which one of the functions is unknown but the others are known. However, we assumed that the unknown function is strongly convex with known convexity parameter, and that we have a region $\mathcal{A}$ where the minimizer of this function lies. We established a necessary condition for a given point to be a minimizer of the sum of known and unknown functions for general compact $\mathcal{A}$. We then considered a special case where the region of the unknown function's minimizer is a ball. In this case, we simplified the necessary condition and provided an algorithm to determine the region that satisfies the necessary condition.  

Future work could focus on providing sufficient conditions for a given point to be a minimizer (to complement our necessary condition). Alternatively, one could analyze properties of the set of solutions that satisfy the necessary condition.

\addtolength{\textheight}{-12cm}   




\bibliographystyle{IEEEtran}
\bibliography{refs2}

\end{document}

%% file: intro.tex
\section{Introduction}

Optimization is an important tool in various fields, including machine learning \cite{friedman2001elements}, signal processing \cite{moon2000mathematical}, control theory, \cite{mayne2000constrained, bryson2018applied, calafiore2013robust},  and robotics \cite{zhu2015distributed, Montijano14, shin1985minimum}. Given an objective function to be optimized, there are several standard algorithms that can be applied to find the optimal variables \cite{boyd2011distributed, beck2009fast, kingma2014adam,hosseini2014online}.

However, in many applications, it may be the case that the objective function is only partially known.  For example, such scenarios are central to the field of robust optimization, where the objective function contains some parametric uncertainty, and the goal is to choose the optimization variable to be robust to the possible realizations of the uncertainty \cite{ben1998robust, bertsimas2011theory, jiang2016solution}.  The problem that we consider in this paper also has a similar flavor, in that we assume that the optimization objective is not fully known.  However, rather than seeking to find a single solution that is simultaneously robust to all possible realizations of the uncertain parameter (or learning that parameter \cite{jiang2016solution}), we instead seek to characterize the region where the minimizer could lie for {\it each} possible realization of the uncertainty.  This approach has the potential to yield  insights regarding the nature of the possible solutions to the given uncertain optimization problem.  

In our recent paper, \cite{kuwaranancharoen2018locationIEEE} we determined a region containing the possible minimizers of a sum of two strongly convex functions, given only the minimizers of the local functions, their strong convexity parameters, and a bound on their gradients. In contrast, in this paper, we shall consider the case of optimizing a sum of known and unknown functions where only limited information about the unknown function is available. 
In this case, we are given some general characteristics of the unknown function, namely a region containing the minimizer, and the strong convexity parameter of the function. Our goal is to determine necessary conditions for a point to be a minimizer of the sum. In particular, we will determine a region where the potential minimizer of the sum can lie.  Thus, if a point from within this region is chosen as an estimate of the true minimizer of the sum, the size of the region can be used to quantify how far the estimate can be from the true minimizer.  Below, we describe an example scenario to illustrate this problem.

\subsection*{An Example Scenario}
In supervised machine learning problems, one uses labeled training data in order to construct a model that can be used to perform regression or classification tasks. The training data consists of pairs $x_i$ and $y_i$ which are the feature vector and label of the $i$-th example, respectively. For simplicity, assume that we have $2$ training sets denoted by $\mathcal{D}_{j} = \{ x_i^{(j)}, y_i^{(j)} \}_{i=1}^{N_j}$ for $j \in \{1, 2 \}$.
We can write the aggregate loss function of the whole dataset $\mathcal{D} = \mathcal{D}_{1} \cup \mathcal{D}_{2}$ as
\begin{align*}
    L(w;  \mathcal{D}) = \underbrace{\sum_{i=1}^{N_1} l(w; \; x_i^{(1)}, y_i^{(1)})}_{L_1(w; \mathcal{D}_1)} + \underbrace{\sum_{i=1}^{N_2} l(w; \; x_i^{(2)}, y_i^{(2)})}_{L_2(w;\mathcal{D}_2)}
\end{align*}
where $w$ is a model parameter that we need to optimize and $l(w; \; x_i^{(j)}, y_i^{(j)})$ is a loss function for each sample. Assume that $L(w;  \mathcal{D})$ is a strongly convex function (which will be the case when we consider linear regression problems or functions incorporating $l_2$ regularization \cite{vapnik2013nature}). Suppose $w^*$ and $w^*_2$ are the minimizer of $L(w;  \mathcal{D})$ and $L_2(w; \mathcal{D}_2)$, respectively.

Now suppose that the entity trying to find the optimal parameter $w$ for $L(w,\mathcal{D})$ can only access the data set $\mathcal{D}_1$, but not $\mathcal{D}_2$ (or alternatively, can only access a corrupted or {\it poisoned} version of $\mathcal{D}_2$ \cite{biggio2012poisoning, mozaffari2015systematic}).  In this case, the entity may only know certain properties of the function $L_2(w; \mathcal{D}_2)$ (such as its general form, convexity parameters, etc.), and a region containing the minimizer of $L_2(w; \mathcal{D}_2)$ (e.g., based on the statistical properties of the underlying data).  Given this limited information about $L_2(w; \mathcal{D}_2)$, and with $L_1(w; \mathcal{D}_1)$ fully known, the entity could seek to find a region that is guaranteed to contain the minimizer of the true function $L(w; \mathcal{D})$.  This is the problem tackled in this paper.